\documentclass{amsart}

\usepackage{graphicx}
\usepackage[all]{xy}
\usepackage{enumitem}
\usepackage{color}
\usepackage{hyperref}
\usepackage{amssymb}

\usepackage{etoolbox}
\apptocmd{\sloppy}{\hbadness 10000\relax}{}{}

\theoremstyle{definition}
\newtheorem{theorem}{Theorem}[section]
\newtheorem{definition}[theorem]{Definition}
\newtheorem{lemma}[theorem]{Lemma}
\newtheorem{corollary}[theorem]{Corollary}
\newtheorem{proposition}[theorem]{Proposition}

\theoremstyle{remark}
\newtheorem{remark}[theorem]{Remark}

\newtheorem*{remark*}{Remark}
\newtheorem*{remarks*}{Remarks}

\numberwithin{equation}{section}

\newcommand{\Z}{\mathbb{Z}}
\newcommand{\Q}{\mathbb{Q}}

\newcommand{\Err}{\mathcal{R}}
\newcommand{\Eff}{\mathcal{F}}
\newcommand{\Mu}{\text{M}}
\newcommand{\rk}{\mathrm{rk}}
\newcommand{\chileq}{\leq_{\chi}}

\begin{document}

\title{Ribbon Cobordisms Between Lens Spaces}

\author{Marius Huber}
\address{Department of Mathematics, Boston College, Chestnut Hill, MA 02467}
\email{marius.huber@bc.edu}
\thanks{The author was partially supported by the Swiss National Science Foundation (grant nr. 191342)}





\begin{abstract}
We determine when there exists a ribbon rational homology cobordism between two connected sums of lens spaces, i.e. one without $3$-handles.
In particular, we show that if a lens space $L$ admits a ribbon rational homology cobordism to a different lens space, then $L$ must be homeomorphic to $L(n,1)$, up to orientation-reversal.
As an application, we classify ribbon $\chi$-concordances between connected sums of $2$-bridge links.
Our work builds on Lisca's work on embeddings of linear lattices.
\end{abstract}

\maketitle


\section{Introduction}\label{intro}

Given oriented rational homology $3$-spheres $Y_{1},Y_{2}$, a \emph{rational homology cobordism} from $Y_{1}$ to $Y_{2}$ is an oriented $4$-manifold $W$ with  boundary $\partial W=-Y_{1}\amalg Y_{2}$, such that the inclusion maps $\iota_{i}\colon Y_{i}\to W$ induce isomorphisms $H_{*}(Y_{i};\Q)\cong H_{*}(W;\Q)$, $i=1,2$.
As an example, if $C=S^{1}\times I\subset S^{3}\times I$ is a smooth concordance from $K_{1}=C\cap S^{3}\times\{0\}$ to $K_{2}=C\cap S^{3}\times\{1\}$, then $\Sigma_{2}(S^{3}\times I,C)$ (the double cover of $S^{3}\times I$ branched along $C$) is a rational homology cobordism from $\Sigma_{2}(S^{3},K_{1})$ to $\Sigma_{2}(S^{3},K_{2})$.
The non-existence of a rational homology cobordism from $\Sigma_{2}(S^{3},K_{1})$ to $\Sigma_{2}(S^{3},K_{2})$ thus obstructs $K_{1}$ and $K_{2}$ from being smoothly concordant.
Endowing the set of rational homology $3$-spheres, considered up to rational homology cobordism, with the operation of connected sum yields the rational homology cobordism group $\Theta_{\Q}^{3}$.
The study of this group gives rise to natural classification problems such as the question of which rational homology $3$-spheres bound rational homology balls \cite{Casson-Harer}, a question which has been answered by Lisca \cite{Lisca1} for the case of lens spaces.

The present paper deals with the following refined notion of rational homology cobordism that first appeared in a paper by Daemi, Lidman, Vela-Vick and Wong \cite{Lidman_etc}.

\begin{definition}\label{rhc}
Suppose that $Y_{1}$ and $Y_{2}$ are rational homology $3$-spheres.
A rational homology cobordism $W$ from $Y_{1}$ to $Y_{2}$ is said to be \emph{ribbon} if $W$ admits a handle decomposition relative to $Y_{1}\times I$ that uses $1$- and $2$-handles only.
If such a cobordism exists, we write $Y_{1}\leq Y_{2}$.
\end{definition}

\begin{remark*}
In the remainder of this paper, we will refer to ribbon rational homology cobordisms simply as \emph{ribbon cobordisms}.
\end{remark*}

Similarly to the above example, one can show that if $C\subset S^{3}\times I$ is a ribbon concordance from $K_{1}$ to $K_{2}$, meaning that $C$ has no local maxima with respect to the second coordinate of $S^{3}\times I$, then $\Sigma_{2}(S^{3}\times I,C)$ is a ribbon cobordism from $\Sigma_{2}(S^{3},K_{1})$ to $\Sigma_{2}(S^{3},K_{2})$.
Recently, Zemke \cite{Zemke} showed that knot Floer homology can be used to obstruct the existence of ribbon concordances.
Zemke's work was in part motivated by a conjecture of Gordon \cite[Conjecture 1.1]{Gordon}, which states that if $K_{1},K_{2}\subset S^{3}$ are knots with the property that there exists a ribbon concordance from $K_{1}$ to $K_{2}$ and vice versa, then $K_{1}$ and $K_{2}$ are isotopic.
Subsequently, several authors have used various versions of knot homologies to prove results that are similar in spirit to Zemke's result (see e.g. \cite{Levine-Zemke}, \cite{JMZ} and \cite{Sarkar}).

In analogy to the study of the rational homology cobordism group, a naturally arising problem is to classify which $3$-manifolds can be obtained by a ribbon cobordism from another $3$-manifold.
In this paper, we address this question for the family of lens spaces.

Before stating our results, we need to review some notation.
The question of when two lens spaces cobound a rational homology cobordism that is not necessarily ribbon has been completely answered by Lisca in a series of two papers \cite{Lisca1, Lisca2}.
More precisely, Lisca defines a set of rational numbers $\Err$ with the property that the lens space $L(p,q)$, $p>q>0$, bounds a rational homology ball (or, equivalently, is rational homology cobordant to $S^{3}$) if and only if $p/q\in\Err$.\footnote{Although we will not use the precise form of $\Err$, we remark that its original definition \cite[Definition 1.1]{Lisca1} is incomplete; see e.g. the footnote in \cite[p. 247]{Lecuona}}
Indeed, Lisca shows that $L(p,q)$, $p>q>0$ with $p/q\in\Err$, bounds a ribbon rational homology ball (\cite[Theorem 1.2]{Lisca1}).
In \cite{Lisca2}, Lisca moreover defines the sets
$$
\Eff_{n}=\left\lbrace\frac{nm^{2}}{nmk+1} \,\middle|\ m>k>0, \gcd(m,k)=1\right\rbrace\subset\Q,\, n\geq 2,
$$
and classifies connected sums of lens spaces that bound a rational homology ball (\cite[Theorem 1.1]{Lisca2}).
We refer the reader to \cite{Lisca2} for a precise statement of the result, but we remark that the sets $\Eff_{n}$, $n\geq 2$, together with the set $\Err$, can be regarded as the building blocks of the statement of that classification.

Building on the machinery that Lisca sets up to prove his classification, we determine when there exists a for ribbon cobordisms between two lens spaces.

\begin{theorem}\label{thm1}
Suppose that $L(p_{1},q_{1})\leq L(p_{2},q_{2})$.
Then, up to simultaneous orientation reversal of $L(p_{1},q_{1})$ and $L(p_{2},q_{2})$, one of the following holds:
\begin{enumerate}
\item $L(p_{1},q_{1})\cong L(p_{2},q_{2})$;
\item $L(p_{1},q_{1})\cong L(n,1)$ and $p_{2}/q_{2}\in\Eff_{n}$, for some $n\geq 2$; or
\item $L(p_{1},q_{1})\cong S^{3}$ and $p_{2}/q_{2}\in\Err$.
\end{enumerate}
Conversely, in each of these cases $L(p_{1},q_{1})\leq L(p_{2},q_{2})$ holds.
\end{theorem}

\begin{remarks*}
\leavevmode
\begin{enumerate}
\item The above theorem can be interpreted as saying that the only lens spaces that admit a ribbon cobordism to a different lens space are the lens spaces of the form $L(n,1)$, for some $n\geq 1$.
\item While the above result makes no statement about the uniqueness of ribbon cobordisms between lens spaces, it follows from a combination of \cite[Lemma 3.5]{Lisca2} and \cite[Corollary 1.3]{Baker-Buck-Lecuona} that cases (2) and (3) above can be realized by a ribbon cobordism which is minimal in the sense that it uses just one $1$-handle and, consequently, one $2$-handle.
\item Using the fact that the double cover of $S^{3}$ branched along a $2$-bridge knot is a lens space (cf. Section \ref{cor_section}), Theorem \ref{thm1} may be used to provide an alternative proof of the fact that \cite[Conjecture 1.1]{Gordon} is true for the family of $2$-bridge knots \cite[Theorem 1.2]{Gordon}.
\end{enumerate}
\end{remarks*}

Based on Theorem \ref{thm1}, we determine which pairs of connected sums of lens spaces cobound a ribbon cobordism.
Before stating the result, we make the following observation.
Suppose that $W$ and $W'$ are ribbon cobordisms from $Y_{1}$ to $Y_{2}$ and from $Y_{1}'$ to $Y_{2}'$, respectively, so that $W$ can be built by attaching $1$- and $2$-handles to $Y_{1}\times I$, and similarly for $W'$.
By attaching corresponding $1$- and $2$-handles to $(Y_{1}\times I)\natural(Y_{1}'\times I)$ outside of the region where the boundary connected sum takes place, we obtain a ribbon cobordism $W''$ from $Y_{1}\#Y_{1}'$ to $Y_{2}\#Y_{2}'$.

\begin{theorem}\label{thm2}
Suppose that $Y_{1}\leq Y_{2}$, where $Y_{i}$ is a finite connected sum of lens spaces, $i=1,2$.
Then there exists a ribbon cobordism $W$ from $Y_{1}$ to $Y_{2}$ that can be decomposed as a boundary connected sum of ribbon cobordisms in such a way that each summand is (possibly orientation-reversingly) homeomorphic to one between the following ordered pairs:
\begin{enumerate}
\item $(L(p,q),L(p,q))$, $p/q>1$;
\item $(L(n,1),L(p,q))$, $p/q\in\Eff_{n}$, for some $n\geq 2$;
\item $(S^{3},L(p,q))$, $p/q\in\Err$;
\item $(S^{3},L(p,p-q)\#L(p,q))$;
\item $(S^{3},L(n,n-1)\#L(p,q))$, $p/q\in\Eff_{n}$, for some $n\geq 2$;
\item $(S^{3},L(p_{1},p_{1}-q_{1})\#L(p_{2},q_{2}))$, $p_{i}/q_{i}\in\Eff_{n}$, $i=1,2$, for some $n\geq 2$; or
\item $(S^{3},L(p_{1},q_{1})\#L(p_{2},q_{2}))$, $p_{i}/q_{i}\in\Eff_{2}$, $i=1,2$.
\end{enumerate}
Conversely, if $(Y_{1},Y_{2})$ is any of the pairs from (1)--(7), then $Y_{1}\leq Y_{2}$ holds.
\end{theorem}

From Theorems \ref{thm1} and \ref{thm2} we derive two corollaries concerning the concordance of $2$-bridge links.
The proofs of those corollaries rely on the fact that any lens space $L(p,q)$ arises as the double cover of $S^{3}$ branched along the $2$-bridge link $K(p,q)$ (see Section \ref{cor_section}).

\subsection*{Outline of the proof and organization}
We give an overview of the argument we use to prove Theorems \ref{thm1} and \ref{thm2}.
Suppose that $W$ is a ribbon cobordism from $L(p,q)$ to $L(r,s)$, so that $\partial W=-L(p,q)\amalg L(r,s)$.
Both $-L(p,q)\cong L(p,p-q)$ and $L(r,s)$ bound positive-definite plumbings $X(p,p-q)$ and $X(r,s)$, respectively (see the paragraph preceding Definition \ref{linlattice}), with intersection lattices given by the linear lattices $\Lambda(p/(p-q))$ and $\Lambda(r/s)$, respectively (see Definition \ref{linlattice}).
By a standard argument using Donaldson's diagonalization theorem \cite[Theorem 1]{Donaldson}, we obtain a full-rank isometric embedding of lattices $\varphi\colon\Lambda(p/(p-q))\oplus\Lambda(r/s)\hookrightarrow\Z^{N}$, where $\Z^{N}$ denotes the standard positive-definite Euclidean lattice (see Section \ref{prelim}).
Crucially, the fact that $W$ is a \emph{ribbon} cobordism translates into the condition $\varphi(\Lambda(p/(p-q)))=\varphi(\Lambda(r/s))^{\perp}$, where the latter denotes the orthogonal complement to $\varphi(\Lambda(r/s))$ in $\Z^{N}$ (Lemma \ref{riblemma}).
Equivalently, $\varphi(\Lambda(p/(p-q)))\subset\Z^{N}$ is a \emph{primitive} sublattice, meaning that the quotient $\Z^{N}/\varphi(\Lambda(p/(p-q))$ contains no torsion.
Furthermore, it follows from Lisca's combinatorial machinery set up in \cite{Lisca1,Lisca2}, that the subset $S\subset\Z^{N}$ consisting of the images of the standard basis elements of $\Lambda(p/(p-q))\oplus\Lambda(r/s)$ can be obtained from a certain `minimal' such subset by repeatedly applying an operation called $2$-final expansion (see Subsection \ref{Embeddings_subs}).
In Lemma \ref{T_emb}, we show that the stable isometry type of $\varphi(\Lambda(r/s))^{\perp}$, in fact, remains unchanged under $2$-final expansions, i.e. $\varphi(\Lambda(r/s))^{\perp}$ only changes by adding orthogonal direct summands isometric to $\Z^{N}$.
Combined with Lisca's classification of lens spaces bounding rational balls, and standard facts about linear lattices, this allows us to deduce Theorem \ref{thm1}.

The proof of Theorem \ref{thm2} relies on Theorem \ref{thm1}, combined with the fact that if a connected sum of lens spaces bounds a rational homology ball, then the corresponding embedding of the orthogonal direct sum of the intersection lattices of the connected summands into $\Z^{N}$ can be decomposed into smaller embeddings which involve at most two of the direct summands (Lemma \ref{reducible}).

It is worth noting that our methods are similar in spirit, but, in a certain sense, complementary to those that Aceto, Celoria and Park used in \cite{ACP}.
Indeed, the proof of their main result relies on investigating the orthogonal complements of subsets $S\subset\Z^{N}$ as above that do not contain so-called bad components (in the terminology of Lisca).
In contrast, the subsets $S\subset\Z^{N}$ we are dealing with in the present article typically do contain bad components.

This paper is organized as follows.
In Section \ref{cor_section}, we prove two naturally derived corollaries to Theorems \ref{thm1} and \ref{thm2} concerning the $\chi$-concordance of $2$-bridge links.
In Section \ref{prelim}, we first review some generalities about integral lattices and some definitions coming from \cite{Lisca1,Lisca2}, after which we proceed to prove a few lemmas concerning embeddings of linear lattices and their connection to ribbon cobordisms.
Finally, in Sections \ref{proof1} and \ref{proof2}, we prove Theorems \ref{thm1} and \ref{thm2}, respectively.

\subsection*{Conventions and notation}

In order to minimize the number of signs, we go by the definition that $L(p,q)$ is the oriented $3$-manifold obtained by performing $p/q$-framed Dehn surgery along the unknot $U\subset S^{3}$.
The mirror image of a link $K\subset S^{3}$ is denoted by $\overline{K}$, and we write $K\simeq K'$ to denote that the links $K$ and $K'$ are isotopic.
Moreover, all manifolds in this paper are assumed to be oriented, so that $-Y$ stands for the oriented manifold obtained from $Y$ by reversing orientation, and we write $Y\cong Y'$ to indicate that $Y$ and $Y'$ are related by an orientation-preserving homeomorphism.

\subsection*{Acknowledgments}
I would like to thank to my advisor, Josh Greene, for suggesting this question and for the countless helpful conversations we had during the course of this project.
I would also like to thank Paolo Lisca for helpful feedback on an earlier version of this paper.
Lastly, I thank Daniele Celoria for useful feedback and his openly available program to compute $d$-invariants of lens spaces, which led to conjecturing Theorem \ref{thm1}.

\section{Applications to the \texorpdfstring{$\chi$}--concordance of 2-bridge links}\label{cor_section}

Recall that the family of $2$-bridge links can be parametrized by pairs of coprime integers $p,q\in\Z$ in such a way that $K(p,p-q)\simeq\overline{K(p,q)}$, $\Sigma_{2}(S^{3},K(p,q))\cong L(p,q)$, and, moreover, that $K(p,q)$ is a knot precisely when $p$ is odd (see e.g. \cite[Chapter 12]{Burde-Zieschang}).
Inspired by \cite[Definition 2]{Donald-Owens}, we make the following definition.
\begin{definition}\label{chi-ribbon}
Let $L_{0},L_{1}\subset S^{3}$ be links, and let $C\subset S^{3}\times I$ be a properly embedded surface satisfying $L_{i}=C\cap S^{3}\times\{i\}$, $i=0,1$.
We say that $C$ is a ribbon $\chi$-concordance from $L_{0}$ to $L_{1}$ if $\chi(C)=0$ and $C$ has no local maxima with respect to the second coordinate of $S^{3}\times I$.
If such a concordance exists, we write $L_{0}\chileq L_{1}$.
\end{definition}

Note that in the case where both $L_{0}$ and $L_{1}$ are knots, this notion coincides with the usual notion of ribbon concordance.
Furthermore, we remark that our definition is more general than \cite[Definition 2]{Donald-Owens}.
Indeed, that definition requires some decorations on the components of the links involved, which are necessary because connected sum is not well-defined for links, and, eventually, to endow the set of $\chi$-concordance classes with a group structure.
The statement of Corollary \ref{cor2}, however, holds true regardless of how one chooses to form the connected sum of the links involved.
Furthermore, \cite[Definition 2]{Donald-Owens} requires a $\chi$-concordance to have no closed components, which becomes redundant as soon as we demand that $C$ be ribbon.
Finally, we point out that we do not require $C$ to be orientable.
In fact, as will become apparent in the proof of Corollary \ref{cor1}, this flexibility is the reason why we choose to go by this definition of concordance, as opposed to one that requires a concordance between links to be a disjoint union of annuli.
Nevertheless, it is easily verified that the double cover of $S^{3}\times I$ branched along a ribbon $\chi$-concordance from $L_{0}$ to $L_{1}$ is a ribbon cobordism from $\Sigma_{2}(S^{3},L_{0})$ to $\Sigma_{2}(S^{3},L_{1})$.

\begin{corollary}\label{cor1}
Let $p_{i},q_{i}\in\Z$ be coprime, $i=1,2$.
Then, possibly after replacing both $K(p_{1},q_{1})$ and $K(p_{2},q_{2})$ by their mirror images, we have $K(p_{1},q_{1})\chileq K(p_{2},q_{2})$ if and only if one of the following holds:
\begin{enumerate}
\item
$K(p_{1},q_{1})\simeq K(p_{2},q_{2})$;
\item
$K(p_{1},q_{1})\simeq K(n,1)$ and $p_{2}/q_{2}\in\Eff_{n}$, for some $n\geq 2$; or
\item
$K(p_{1},q_{1})\simeq U$ and $p_{2}/q_{2}\in\Err$.
\end{enumerate}
\end{corollary}

\begin{proof}
If $K(p_{1},q_{1})\chileq K(p_{2},q_{2})$, then we have $L(p_{1},q_{1})\leq L(p_{2},q_{2})$.
Therefore, one of (1)--(3) from Theorem \ref{thm1} holds, and the claim follows.

Conversely, it is clear that $K(p_{1},q_{1})\chileq K(p_{2},q_{2})$ holds if $K(p_{1},q_{1})\simeq K(p_{2},q_{2})$.
Suppose that $K(p_{1},q_{1})\simeq K(n,1)$ and $p_{2}/q_{2}\in\Eff_{n}$, for some $n\geq 2$.
By \cite[Lemma 3.5]{Lisca2}, there exists a ribbon move turning $K(p_{2},q_{2})$ into a split link consisting of $K(n,1)$ and an unknot.
Capping off the unknot with a disk yields the desired ribbon $\chi$-concordance.
If $p_{2}/q_{2}\in\Err$, then, by the main theorem of \cite{Lisca1}, $K(p_{2},q_{2})$ bounds a properly embedded ribbon surface $C\subset B^{4}$ that is homeomorphic either to a disk or to the disjoint union of a disk with a M{\"o}bius band, depending on whether $K(p_{2},q_{2})$ is a knot or a link, respectively.
In either case, $\chi(C)=1$, so by removing a small disk from $C$ we obtain a ribbon $\chi$-concordance from $U$ to $K(p_{2},q_{2})$.
\end{proof}

Before stating the analogous corollary to Theorem \ref{thm2}, we observe that if $C$ and $C'$ are ribbon $\chi$-concordances from $K_{1}$ to $K_{2}$ and from $K_{1}'$ to $K_{2}'$, respectively, where $K_{i},K_{i}'\subset S^{3}$ are links, $i=1,2$, we can sum $C$ and $C'$ together along properly embedded intervals $J\subset C$, $J'\subset C'$ that are transverse to $S^{3}\times \{t\}\subset S^{3}\times I$ for all $t\in I$.
This yields a ribbon $\chi$-concordance $C''$ from $K_{1}\#K_{1}'$ to $K_{2}\#K_{2}'$.

\begin{corollary}\label{cor2}
Suppose that $K_{1}\chileq K_{2}$, where $K_{i}$ is a finite connected sum of $2$-bridge links, $i=1,2$.
Then there exists a ribbon $\chi$-concordance $C$ from $K_{1}$ to $K_{2}$ that can be decomposed as a sum of ribbon $\chi$-concordances in such a way that each summand is (possibly after mirroring) a ribbon $\chi$-concordance between one of the following ordered pairs:
\begin{enumerate}
\item $(K(p,q),K(p,q))$, $p/q>1$;
\item $(K(n,1),K(p,q))$, $p/q\in\Eff_{n}$, for some $n\geq 2$;
\item $(U,K(p,q))$, $p/q\in\Err$;
\item $(U,K(p,p-q)\#K(p,q))$;
\item $(U,K(n,n-1)\#K(p,q))$, $p/q\in\Eff_{n}$, for some $n\geq 2$;
\item $(U,K(p_{1},p_{1}-q_{1})\#K(p_{2},q_{2}))$, $p_{i}/q_{i}\in\Eff_{n}$, $i=1,2$, for some $n\geq 2$; or
\item $(U,K(p_{1},q_{1})\#K(p_{2},q_{2}))$, $p_{i}/q_{i}\in\Eff_{2}$, $i=1,2$.
\end{enumerate}
Conversely, if $(K_{1},K_{2})$ is any of the pairs from (1)--(7), then $K_{1}\leq K_{2}$ holds.
\end{corollary}

\begin{proof}
Let $C\subset S^{3}\times I$ be a ribbon $\chi$-concordance as in the statement of the theorem, so that $W=\Sigma_{2}(S^{3}\times I,C)$ is a ribbon cobordism from $Y_{1}$ to $Y_{2}$, where $Y_{i}=\Sigma_{2}(S^{3},K_{i})$ is a connected sum of lens spaces of corresponding parameters, $i=1,2$.
By Theorem \ref{thm2}, we may assume that $W$ is a boundary connected sum of the ribbon cobordisms listed there.
Together with the fact that $\Sigma_{2}(S^{3},K\#K')\cong\Sigma_{2}(S^{3},K)\#\Sigma_{2}(S^{3},K')$, this implies that $C$ must be of the desired form.

Conversely, it suffices to show that each of the cases (1)--(7) can be realized by a ribbon $\chi$-concordance.
Cases (1)--(3) follow from Corollary \ref{cor1}.
For case (4), note that $K(p,p-q)\amalg K(p,q)\simeq\overline{K(p,q)}\amalg K(p,q)\subset S^{3}$ bounds a disjoint union of $m$ annuli, properly embedded in $B^{4}$, where $m$ equals either $1$ or $2$, depending on whether $K(p,q)$ is a knot or a link, respectively.
Moreover, this disjoint union of annuli can be chosen not to have any local maxima with respect to the radial distance function on $B^{4}$.
It follows that $K(p,p-q)\#K(p,q)$ bounds a disjoint union of a disk and, possibly, an annulus with the same property.
Puncturing the disk yields a ribbon $\chi$-concordance from $U$ to $K(p,p-q)\#K(p,q)$.
For case (5), note that, by \cite[Lemma 3.5]{Lisca2} again, $K(p,q)$ can be turned into $K(n,1)$ by the reverse of a ribbon $\chi$-concordance, which, using case (4), shows that $U\chileq K(n,n-1)\#K(n,1)\chileq K(n,n-1)\#K(p,q)$.
The remaining cases are handled similarly.
\end{proof}

\section{Preliminaries}\label{prelim}

\subsection{Lattices}

An \emph{integral lattice} is a free Abelian group $\Lambda$ endowed with a symmetric, bilinear pairing $\langle\cdot,\cdot\rangle\colon\Lambda\times\Lambda\to\Z$.
We usually write $x\cdot y$ to mean $\langle x,y\rangle$, $x,y\in\Lambda$.

Two lattices $\Lambda_{1},\Lambda_{2}$ are \emph{isometric} if there exists an isomorphism $\varphi\colon\Lambda_{1}\to\Lambda_{2}$ that preserves the bilinear pairings, and we write $\Lambda_{1}\cong\Lambda_{2}$.
Moreover, we say that $\Lambda_{1}$ and $\Lambda_{2}$ are \emph{stably isometric} (denoted by $\Lambda_{1}\simeq\Lambda_{2}$) if $\Lambda_{1}\cong\Lambda_{2}\oplus\Z^{k}$ or $\Lambda_{1}\oplus\Z^{k}\cong\Lambda_{2}$ for some $k\geq 0$, where $\Z^{n}$ denotes the standard Euclidean lattice with basis $\{e_{1},\dots,e_{n}\}$ and pairing given by $e_{i}\cdot e_{j}=\delta_{ij}$.
To any integral lattice $\Lambda$ one can associate its \emph{dual lattice}
$\Lambda^{*}:=\{\xi\in\Lambda\otimes\Q\mid\langle\xi,x\rangle\in\Z\text{ for all }x\in\Lambda\}$, where $\langle\cdot,\cdot\rangle$ is extended to $\Lambda\otimes\Q$ by $\Q$-bilinearity.

Given a sublattice $\Lambda$ of $\Mu$, we define the \emph{orthogonal complement} of $\Lambda$ in $\Mu$ as $\Lambda^{\perp}=\{x\in\Mu\mid x\cdot y=0\text{ for all }y\in\Lambda\}$.\footnote{Throughout this paper, $\Mu$ denotes a capitalized $\mu$, whereas $M$ stands for a capitalized $m$, as usual.}
Finally, we say that a sublattice $\Lambda\subset\Mu$ is \emph{primitive} if the quotient $\Mu/\Lambda$ is a free Abelian group.
We establish two equivalent characterizations of primitivity that will be used later in the paper.
This should be well-known to experts, and we provide a proof for completeness.
\begin{lemma}\label{primitivity}
Let $\Mu$ be an integral lattice and $\Lambda\subset\Mu$ a sublattice.
Then the following are equivalent.
\begin{enumerate}
\item $\Mu/\Lambda$ is torsion-free.
\item The natural restriction map $r\colon\Mu^{*}\to\Lambda^{*}$ is surjective.
\item $(\Lambda\otimes\Q)\cap\Mu=\Lambda$.
\end{enumerate}
\end{lemma}

\begin{proof}
(1)$\Leftrightarrow$(2):
Set $n=\rk\Mu$, $m=\rk\Lambda$, and represent the embedding $\Lambda\subset\Mu$ by a matrix $A\in\Z^{n\times m}$ with respect to some bases of $\Mu$ and $\Lambda$.
Then, after endowing $\Mu^{*}$ and $\Lambda^{*}$ with their respective dual bases, the restriction $r\colon\Mu^{*}\to\Lambda^{*}$ is represented by the matrix $A^{\top}\in\Z^{m\times n}$.
Now, by Smith normal form, $A^{\top}$ is surjective iff $A^{\top}=P[E_{m}|0]Q$ for some invertible $P\in\Z^{m\times m},Q\in\Z^{n\times n}$, where $E_{m}$ denotes the identity matrix of size $m\times m$.
Thus, $A^{\top}$ is surjective iff $A=Q^{\top}\left[\frac{E_{m}}{0}\right]P^{\top}$, which, by Smith normal form again, is in turn equivalent to $\Mu/\Lambda$ being torsion-free.

(1)$\Leftrightarrow$(3):
Suppose $(\Lambda\otimes\Q)\cap\Mu=\Lambda$, and assume for the sake of a contradiction that $\Mu/\Lambda$ contains torsion.
Then there exists some $x\in\Mu\setminus\Lambda$ with the property that $kx\in\Lambda$ for some $k\geq 2$, which implies that $x\in(\Lambda\otimes\Q)\cap\Mu\subset\Lambda$, a contradiction.

Conversely, suppose that $\Mu/\Lambda$ is torsion-free.
Because $\Lambda\subset(\Lambda\otimes\Q)\cap\Mu$, it suffices to show that $(\Lambda\otimes\Q)\cap\Mu\subset\Lambda$.
Pick $x\in(\Lambda\otimes\Q)\cap\Mu$.
Then there exists $k\geq 2$ such that $kx\in\Lambda\cap\Mu=\Lambda$.
Since $\Mu/\Lambda$ is free, we must therefore have $x\in\Lambda$, too.
\end{proof}

\begin{remark}\label{rem_primitivity}
For later reference, we point out that, using characterization (3) from Lemma \ref{primitivity}, it is easily verified that the orthogonal complement $\Lambda^{\perp}$ to any sublattice $\Lambda\subset\Mu$ is a primitive sublattice, and, moreover, that the only full-rank primitive sublattice of a lattice $\Mu$ is, in fact, $\Mu$ itself.
\end{remark}

Suppose now that $X$ is a $4$-manifold that is a plumbing of disk bundles over $S^{2}$ along a weighted tree $\Gamma$ with $n$ vertices.
That is, the weight of a vertex is the normal Euler number of the corresponding disk bundle, and two disk bundles are plumbed together precisely when the corresponding vertices are connected by an edge in $\Gamma$.
Suppose further that $Y=\partial X$ is a rational homology $3$-sphere.
Then the intersection form $Q_{X}$ endows $H_{2}(X;\Z)\cong\Z^{n}$ with the structure of an integral lattice $\Lambda$, and, moreover, the dual lattice is given by $\Lambda^{*}=H^{2}(X;\Z)$ (see e.g. \cite[Sections 1.2 and 6.1]{Gompf-Stipsicz}).

In the case where $Y=L(p,q)$, one can choose $X$ to be the plumbing along a linear graph with weights $a_{1},\dots,a_{n}$, where the $a_{i}\geq 2$ are such that
$$
[a_{1},\dots,a_{n}]^{-}:=a_{1}-\dfrac{1}{a_{2}-\dfrac{1}{\cdots-\dfrac{1}{a_{n}}}}=\dfrac{p}{q}.
$$
We denote this plumbing by $X(p,q)$.

\begin{definition}\label{linlattice}
A lattice $\Lambda$ is a \emph{linear lattice} if it admits a basis $\{v_{1},\dots,v_{n}\}$ such that
\begin{equation}
\label{stdbasis}
v_{i}\cdot v_{j}=
\begin{cases}
a_{i}, & \text{if } i=j,\\
0\text{ or }1, & \text{if } |i-j|=1,\\
0, & \text{if } |i-j|>1,
\end{cases}
\end{equation}
for some $a_{1},\dots,a_{n}\geq 2$.
In the case where $v_{i}\cdot v_{j}=1$ whenever $|i-j|=1$, we write $\Lambda=\Lambda(a_{1},\dots,a_{n})$ or $\Lambda=\Lambda(p/q)$, where $p/q=[a_{1},\dots,a_{n}]^{-}$.
Moreover, we will write $\Lambda(\dots,2^{[k]},\dots)=\Lambda(\dots,\underbrace{2,\dots,2}_{k},\dots)$.
\end{definition}

\begin{remark}\label{linlattice_rem}
We state without proof two facts about linear lattices that we implicitly use throughout this paper.
\begin{enumerate}
\item Given a linear lattice $\Lambda(p/q)$, $p/q>1$, the integers $a_{1},\dots,a_{n}$ are uniquely determined by the conditions $a_{i}\geq 2$, $i=1,\dots,n$, and $[a_{1},\dots,a_{n}]^{-}=p/q$ (see e.g. \cite[Section 5.2]{Gompf-Stipsicz}).
\item For $p/q>1$, the intersection lattice of $X(p,q)$ is isometric to $\Lambda(p/q)$.
In fact, $\partial X(p,q)\cong\partial X(r,s)$ if and only if $\Lambda(p/q)\cong\Lambda(r/s)$ (see e.g. \cite[Section 1.5]{Saveliev}, and \cite[Theorem 3]{Gerstein} and \cite[Proposition 3.6]{Josh_LRP}).
\end{enumerate}
\end{remark}

In what follows, we will always endow a linear lattice $\Lambda(a_{1},\dots,a_{n})$ with the standard basis $\{v_{1},\dots,v_{n}\}$ satisfying \eqref{stdbasis}, and similarly for orthogonal direct sums of linear lattices.

\subsection{Embeddings of linear lattices}\label{Embeddings_subs}

We now briefly recall a few definitions from \cite{Lisca1} and \cite{Lisca2} that will play a role later in the paper.

\begin{definition}\label{big_def}
Let $S=\{v_{1},\dots,v_{N}\}\subset\Z^{N}.$
\begin{enumerate}
\item
$S$ is called a \emph{linear subset} if its elements satisfy
$$
v_{i}\cdot v_{j}=
\begin{cases}
a_{i}, & \text{if } i=j,\\
0\text{ or }1, & \text{if } |i-j|=1,\\
0, & \text{if } |i-j|>1,
\end{cases}
$$
for some $a_{i}\geq 2$.
\item
If $S$ is a linear subset, its \emph{intersection graph} is defined as the graph with one vertex for each element $v_{i}\in S$ and an edge $(v_{i},v_{j})$ precisely if $v_{i}\cdot v_{j}=1$.
The number of connected components of the intersection graph is denoted by $c(S)$.
\item
Two elements $v,w\in\Z^{N}$ are said to be \emph{linked} if there exists an index $k\in\{1,\dots,N\}$ such that $e_{k}\cdot v\neq 0$ and $e_{k}\cdot w\neq 0$.
Moreover, a linear subset $S$ is called \emph{irreducible} if for any $v,w\in S$ there exist $v_{1},\dots,v_{m}\in S$ such that $v_{1}=v$, $v_{m}=w$, and $v_{i}$ and $v_{i+1}$ are linked, $i=1,\dots,m-1$.
If $S$ is not irreducible, it is called \emph{reducible}.
\item
Let $S=\{v_{1},\dots,v_{N}\}$ be a linear subset such that $|v_{i}\cdot e_{j}|\leq 1$ for all $i,j \in\{1,\dots,N\}$, and suppose that there exist $h,s,t\in\{1,\dots,N\}$ such that $v_{t}\cdot v_{t}>2$ and $e_{h}\cdot v_{j}\neq 0$ if and only if $j\in\{s,t\}$.
Define $S'\subset\langle e_{1},\dots,e_{h-1},e_{h+1},\dots,e_{N}\rangle$ by  $S':=S\setminus\{v_{s},v_{t}\}\cup\{v_{t}-(e_{h}\cdot v_{t})e_{h}\}.$
Then $S'$ is said to be obtained from $S$ by a \emph{contraction}, and, conversely, $S$ is said to be obtained from $S'$ by an \emph{expansion}.
\item
If, in addition to the hypotheses of (4), both the vertices of the intersection graph of $S$ corresponding to $v_{s}$ and $v_{t}$ have degree $1$, and, moreover, $v_{s}\cdot v_{s}=2$, we say that $S'$ is obtained from $S$ by a \emph{$2$-final contraction}, and that $S$ is obtained from $S'$ by a \emph{$2$-final expansion}.
\item
Let $S'=\{v_{1},\dots,v_{N}\}\subset\Z^{N}$ be a linear subset and suppose that there exists $1<t<N$ such that $C'=\{v_{t-1},v_{t},v_{t+1}\}$ is a connected component of the intersection graph of $S'$ satisfying $v_{t-1}\cdot v_{t-1}=v_{t+1}\cdot v_{t+1}=2$, $v_{t}\cdot v_{t}>2$ and $\{i\mid v_{i}\cdot e_{j}\neq 0\}=\{t-1,t,t+1\}$, for some $1\leq j\leq N$.
Let $S\subset\Z^{M}$ be a subset of cardinality $M\geq N$ obtained from $S'$ by applying a sequence of $2$-final expansions to the connected component $C'$ of $S'$.
Then the connected component $C$ of of the intersection graph of $S$ that naturally corresponds to $C'$ is said to be a \emph{bad component} of $S$.
The number of bad components of $S$ is denoted by $b(S)$.
\end{enumerate}
\end{definition}

A few remarks about these definition are in order.
First, note that a subset $S\subset\Z^{N}$ is linear if and only if $\langle S\rangle\subset\Z^{N}$ is a linear lattice (with the pairing induced by that of $\Z^{N}$).
Moreover, it is not hard to see that if $S'$ is obtained from a linear subset $S$ by a contraction or an expansion, then $S'$ is a linear subset as well.
Lastly, the notion of a bad component will be crucial in the proof of Theorem \ref{thm1}.
Loosely speaking, if $S=\{v_{1},\dots,v_{N}\}\subset\Z^{N}$ is a linear subset corresponding to an isometric embedding $\Lambda\hookrightarrow\Z^{N}$, then a bad component $C=\{v_{t-1},v_{t},v_{t+1}\}$ of $S$ with $v_{t}\cdot v_{t}=n+1$, $n\geq 2$, corresponds to a direct summand of $\Lambda$ of the form $\Lambda(p/q)$, $p/q\in\Eff_{n}$.

We conclude this subsection by proving two lemmas that will be used in the proofs of the main theorems.
We point out that the results of these lemmas are not entirely novel, but are rather reformulations of results that are implicit in \cite{Lisca1,Lisca2}.

The first lemma deals with the orthogonal complement to a bad component of a linear subset.

\begin{lemma}\label{T_emb}
Let $S=\{v_{1},\dots,v_{N}\}\subset\Z^{N}$ be a linear subset and suppose that the intersection graph of $S$ has a bad component $C=\{v_{t-1},v_{t},v_{t+1}\}$, so that $v_{t}\cdot v_{t}=m+1$, for some $m\geq 2$.
Then, with respect to some orthonormal basis of $\Z^{n}$, we have that
\begin{equation*}
C=\langle e_{m+1}+e_{m+2},e_{1}+\cdots+e_{m+1},e_{m+1}-e_{m+2}\rangle\subset\Z^{N}.
\end{equation*}
Moreover, if $C'\subset\Z^{N+K}$ is a linear subset that is obtained from $C$ by a sequence of $K$ $2$-final expansions, $K\geq 0$, then $\langle C'\rangle^{\perp}\simeq \Lambda(m/(m-1))$.
\end{lemma}

\begin{proof}
For the first part, note that, by definition of a bad component, the coefficients of $v_{i}$ are at most $1$ in absolute value, $i\in\{t-1,t,t+1\}$.
We may thus assume that $v_{t}=e_{1}+\cdots+e_{m+1}$.
Using the facts $\{i\mid v_{i}\cdot e_{j}\neq 0\}=\{t-1,t,t+1\}$ for some $1\leq j\leq N$ and $v_{t-1}\cdot v_{t+1}=0$, it is then readily checked that, up to change of orthonormal basis of $\Z^{N}$, it must be the case that $v_{t-1}=e_{m+1}+e_{m+2}$ and $v_{t+1}=e_{m+1}-e_{m+2}$.

For the second part, if $K=0$, so that $C'=C$, it is easily verified that
\begin{equation*}
\langle C\rangle^{\perp}=\langle e_{1}-e_{2},\dots,e_{m-1}-e_{m},e_{m+3},\dots,e_{N}\rangle\simeq\Lambda(2^{[m-1]})\cong\Lambda(m/(m-1)).
\end{equation*}
Consider the case where $K>0$, and suppose that $C$ is given as above.
It is easy to see that there are only two ways of applying a $2$-final expansion to $C$, which (up to change of orthonormal basis of $\Z^{N}$) are given by replacing
$$\langle e_{m+1}+e_{m+2},e_{1}+\cdots+e_{m+1},e_{m+1}-e_{m+2}\rangle\subset\Z^{N}$$
by either
$$
\langle e_{m+2}+e_{m+3},e_{m+1}+e_{m+2},e_{1}+\cdots+e_{m+1},e_{m+1}-e_{m+2}+e_{m+3}\rangle\subset\Z^{N+1}
$$
or
$$
\langle e_{m+1}+e_{m+2}+e_{m+3},e_{1}+\cdots+e_{m+1},e_{m+1}-e_{m+2},-e_{m+2}+e_{m+3}\rangle\subset\Z^{N+1},
$$
and similarly for potential subsequent $2$-final expansions.
Hence
\begin{equation*}
\langle C'\rangle^{\perp}=\langle e_{1}-e_{2},\dots,e_{m-1}-e_{m},e_{m+K+3},\dots,e_{N+K}\rangle\simeq\Lambda(2^{[m-1]})\cong\Lambda(m/(m-1)),
\end{equation*}
and the claim follows.
\end{proof}

The following lemma pertaining to reducible linear subsets will be used in the proof of Theorem \ref{thm2}, where it will allow us to reduce a full-rank embedding of a linear lattice into smaller embeddings.
To clarify the hypotheses of the lemma, note that if a connected sum of lens spaces $L(p_{1},q_{1})\#\cdots\#L(p_{n},q_{n})$ bounds a rational ball $W$, we not only obtain a full-rank isometric embedding
$$\Lambda(p_{1}/q_{1})\oplus\cdots\oplus\Lambda(p_{n}/q_{n})\hookrightarrow\Z^{N}$$
(cf. the proof of Lemma \ref{riblemma}), but, by considering $-W$, we additionally obtain a full-rank isometric embedding
$$\Lambda(p_{1}/(p_{1}-q_{1}))\oplus\cdots\oplus\Lambda(p_{n}/(p_{n}-q_{n}))\hookrightarrow\Z^{N'}.$$
We remark that the existence of both embeddings is necessary for the conclusion of the lemma to hold.

\begin{lemma}\label{reducible}
Let $\Lambda=\Lambda(p_{1}/q_{1})\oplus\cdots\oplus\Lambda(p_{n}/q_{n})$ be a linear lattice, set $\Lambda'=\Lambda(p_{1}/(p_{1}-q_{1}))\oplus\cdots\oplus\Lambda(p_{n}/(p_{n}-q_{n}))$, and suppose that there exist full-rank isometric embeddings $\varphi\colon\Lambda\hookrightarrow\Z^{N}$ and $\varphi'\colon\Lambda'\hookrightarrow\Z^{N'}$.

Then, after possibly permuting $p_{1}/q_{1},\dots,p_{n}/q_{n}$ and switching the roles of $\Lambda$ and $\Lambda'$, $\varphi$ can be decomposed as $\varphi=\varphi_{1}\oplus\widetilde{\varphi}$ in such a way that the linear subset $S_{1}\subset\Z^{N}$ corresponding to $\varphi_{1}$ is irreducible, and $\varphi_{1}$ and $\widetilde{\varphi}$ are full-rank isometric embeddings of one of the following forms:
\begin{enumerate}
\item
$\varphi_{1}\colon\Lambda(p_{1}/q_{1})\hookrightarrow\Z^{N_{1}}$ and $\widetilde{\varphi}\colon\Lambda(p_{2}/q_{2})\oplus\cdots\oplus\Lambda(p_{n}/q_{n})\hookrightarrow\Z^{N_{2}}$, $N_{1}+N_{2}=N$;
\item
$\varphi_{1}\colon\Lambda(p_{1}/q_{1})\oplus\Lambda(p_{2}/q_{2})\hookrightarrow\Z^{N_{1}}$ and $\widetilde{\varphi}\colon\Lambda(p_{3}/q_{3})\oplus\cdots\oplus\Lambda(p_{n}/q_{n})\hookrightarrow\Z^{N_{2}}$, $N_{1}+N_{2}=N$.
\end{enumerate}
\end{lemma}
\begin{proof}
Let $S\subset\Z^{N}$ be the linear subset corresponding to $\varphi$ and write $S=S_{1}\cup\cdots\cup S_{m}$, where the $S_{i}$ are the maximal irreducible subsets of $S$, so that each $S_{i}$ corresponds to the orthogonal direct sum of some of the $\Lambda(p_{1}/q_{1}),\dots,\Lambda(p_{n}/q_{n})$, $i\in\{1,\dots,m\}$.
Since no $v\in S_{i}$ is linked to any $w\in S_{j}$, $i\neq j$, $\varphi$ can be decomposed as an orthogonal direct sum of isometric embeddings $\varphi=\varphi_{1}\oplus\cdots\oplus\varphi_{m}$ such that the linear subset corresponding to $\varphi_{i}$ is $S_{i}\subset\Z^{N}$, $i\in\{1,\dots,m\}$.
Moreover, we can view each $\varphi_{i}$ as a full-rank isometric embedding into $\Z^{N_{i}}\subset\Z^{N}$, where $\Z^{N_{i}}=\langle e_{i}\mid v\cdot e_{i}\neq 0\text{ for some }v\in S_{i}\rangle$.
Indeed, we have $N_{i}=|S_{i}|$, $i\in\{1,\dots,m\}$, by linear independence of the elements of $S$ (cf. \cite[Remark 2.1]{Lisca1}).
It now follows from \cite[Proposition 4.10]{Lisca2} that, up to reordering the $S_{1},\dots,S_{m}$ and the $p_{1}/q_{1},\dots,p_{n}/q_{n}$ and switching the roles of $\Lambda$ and $\Lambda'$, the linear subset $S_{1}\subset\Z^{N_{1}}$ corresponds to a full-rank isometric embedding that is either of the form $\varphi_{1}\colon\Lambda(p_{1}/q_{1})\hookrightarrow\Z^{N_{1}}$, or $\varphi_{1}\colon\Lambda(p_{1}/q_{1})\oplus\Lambda(p_{2}/q_{2})\hookrightarrow\Z^{N_{1}}$.
Indeed, for that proposition to apply, we must ensure that the quantity $I(S_{1})+b(S_{1})$ is negative, which by \cite[Lemma 5.3]{Lisca2} can always be achieved by possibly considering $\varphi'$ instead of $\varphi$ (we refer the reader to \cite[Definition 2.3]{Lisca1} for the definition of $I(S)$).
Setting $\widetilde{\varphi}=\varphi_{2}\oplus\cdots\oplus\varphi_{m}$, we have that $\varphi=\varphi_{1}\oplus\widetilde{\varphi}$ is of the desired form.
\end{proof}

\subsection{Ribbon cobordisms and homology}\label{ribhomo}

In this subsection we state and prove two lemmas that give two conditions for a rational homology cobordism to be ribbon.
The first lemma deals with the order of the homology groups involved, whereas the second one gives a condition on the lattices involved.

\begin{lemma}\label{homlemma}
Let $L_{1}$ and $L_{2}$ be lens spaces, and suppose $W$ is a ribbon cobordism from $L_{1}$ to $L_{2}$.
Then $|H_{1}(L_{2};\Z)|=u^{2}\cdot|H_{1}(L_{1};\Z)|$ for some $u\geq 1$.
\end{lemma}
\begin{proof}
Set $n_{i}=|H_{1}(L_{i};\Z)|$.
By \cite[Proposition 2.1]{Lidman_etc}, we have that
$$
\pi_{1}(L_{1})\hookrightarrow\pi_{1}(W)\twoheadleftarrow\pi_{1}(L_{2}).
$$
Since $L_{i}$ is a lens space, we have that $\pi_{1}(L_{i})=H_{1}(L_{i};\Z)$, $i=1,2$, and it follows that $n_{1}|n_{2}$, i.e. $n_{2}=kn_{1}$ for some $k\geq 1$.
Moreover, since $-L_{1}\#L_{2}$ bounds a rational homology ball, we must have that $n_{1}n_{2}=m^{2}$, for some $m\geq 1$ \cite[Lemma 3]{Casson-Gordon}.
Thus, $n_{1}n_{2}=kn_{1}^{2}=m^{2}$, so $k$ must be a perfect square and the claim follows.
\end{proof}

\begin{remark}
By a simpler version of the argument used in the proof of \cite[Proposition 2.1]{Lidman_etc} (which is modeled on the one used in the proof of \cite[Lemma 3.1]{Gordon}), one can show that the conclusion of Lemma \ref{homlemma} continues to hold when $L_{i}$ is replaced by any rational homology $3$-sphere $Y_{i}$, $i=1,2$.
\end{remark}

\begin{lemma}\label{riblemma}
Let $Y_{1},Y_{2}$ be oriented rational homology spheres such that $-Y_{1}$ and $Y_{2}$ bound positive-definite $4$-manifolds $X_{1}$ and $X_{2}$, respectively, with $H_{1}(X_{2};\Z)=0$, and let $\Lambda_{i}$ denote the intersection lattice of $X_{i}$, $i=1,2$.
If $Y_{1}\leq Y_{2}$, then there exists a full-rank isometric embedding $\varphi\colon\Lambda_{1}\oplus\Lambda_{2}\hookrightarrow \Z^{N}$ with the property $\varphi(\Lambda_{1})=\varphi(\Lambda_{2})^{\perp}$.
\end{lemma}
\begin{proof}
The first half of the statement is a consequence of the following standard argument.
Let $W$ be a ribbon cobordism from $Y_{1}$ to $Y_{2}$, and define $Z=X_{1}\cup_{-Y_{1}}W\cup_{Y_{2}} X_{2}$.
Then $Z$ is a closed, orientable, positive-definite $4$-manifold, so by \cite[Theorem 1]{Donaldson} has intersection form isometric to $\Z^{N}$.
We thus obtain a full-rank embedding of intersection lattices $\Lambda_{1}\oplus \Lambda_{2}\hookrightarrow \Z^{N}$.
This proves the first half of the claim.

For the second half, note that $\varphi(\Lambda_{1})$ is a full-rank sublattice of $\varphi(\Lambda_{2})^{\perp}$.
By Remark \ref{rem_primitivity}, it thus suffices to show that $\varphi(\Lambda_{1})\subset\Z^{N}$ is a primitive sublattice.
To this end, note that $(\Z^{N})^{*}=H^{2}(Z;\Z)$ and $\varphi(\Lambda_{1})^{*}=H^{2}(X_{1};\Z)$, and let $r_{1}\colon H^{2}(Z;\Z)\to H^{2}(X_{1};\Z)$ denote the restriction map.
Consider the following portion of the long exact sequence in cohomology of the pair $(Z,X_{1})$:
$$
\xymatrix{
H^{2}(Z;\Z) \ar[r]^{r_{1}} & H^{2}(X_{1};\Z)  \ar[r] & H^{3}(Z,X_{1};\Z)\\  
}
$$
By excision and Poincar{\'e} duality, respectively, we have that
$$H^{3}(Z,X_{1};\Z)\cong H^{3}(W\cup_{Y_{2}}X_{2},Y_{1};\Z)\cong H_{1}(W\cup_{Y_{2}}X_{2};\Z)=0,$$
where the latter term vanishes because $H_{1}(X_{2};\Z)=0$ and $W\cup_{Y_{2}}X_{2}$ can be built from $X_{2}$ without using $1$-handles (as seen by turning the cobordism $W$ upside down).
Hence $r_{1}$ surjects and $\varphi(\Lambda_{1})\subset\Z^{N}$ is a primitive sublattice by Lemma \ref{primitivity}.
\end{proof}

\section{The proof of Theorem \ref{thm1}}\label{proof1}

The main ingredient to the proof of Theorem \ref{thm1} is the following proposition, which essentially deals with the case where the linear subset $S\subset\Z^{N}$ coming from a ribbon cobordism contains bad components.

\begin{proposition}\label{key}
Let $W$ be a ribbon cobordism from  $L(p,q)$ to $L(r,s)$, where $r/s\in\Eff_{n}$, for some $n\geq 2$, and $p\neq r$.
Suppose further that the linear subset $S\subset\Z^{N}$ associated to the corresponding embedding $\varphi\colon\Lambda(p/(p-q))\oplus\Lambda(r/s)\hookrightarrow\Z^{N}$ is irreducible.
Then we must have that $L(p,q)\cong L(n,1)$.
\end{proposition}
\begin{proof}
Set $L_{1}=L(p,q)$, $L_{2}=L(r,s)$ and $\Lambda_{1}=\Lambda(p/(p-q))$, $\Lambda_{2}=\Lambda(r/s)$.
Write $S=S_{1}\cup S_{2}$, where $S_{i}\subset\Z^{N}$ is the linear subset corresponding to $\varphi(\Lambda_{i})$, so that $\Lambda_{i}=\langle S_{i}\rangle$, $i=1,2$.

Since $r/s\in\Eff_{n}$, for some $n\geq 2$, the main result of \cite{Lisca2} implies that we must have that either $p/q=n/1$, or that either $p/q$ or $p/(p-q)$ belongs to $\Eff_{n}$ (where $p/(p-q)\in\Eff_{n}$ can only happen if $n=2$).
If $p/q=n/1$ we are done, so we assume that $L_{1}$ is not homeomorphic to a lens space of the form $L(n,1)$.
In what follows, we determine the stable isometry type of $\varphi(\Lambda_{1})=\varphi(\Lambda_{2})^{\perp}$, which, by Remark \ref{linlattice_rem}, determines the oriented homeomorphism type of $L_{1}$.

Since $c(S)=2$, it follows from \cite[Lemma 5.2]{Lisca2} that $b(S)\in\{0,1,2\}$.
If $b(S)=0$, then by the first subcase of the proof in \cite[p. 2160]{Lisca2}, we have that $L_{1}\cong L_{2}$, which contradicts our assumption on $p$ and $r$.
Therefore, we must have that $b(S)\in\{1,2\}$.
Suppose first that $b(S)=1$, so that either $S_{1}$ or $S_{2}$ is a bad component.
Using the facts $L_{1}\ncong L(n,1)$ and $p\neq r$, it follows from the second subcase of the proof in \cite[p. 2160]{Lisca2}, that, after possibly replacing $W$ by $-W$, the bad component of $S$ is, in fact, $S_{2}$ and, moreover, that $S_{2}$ admits a sequence of $2$-final contractions $S_{2}\searrow\cdots\searrow C$, where $C=\{v_{t-1},v_{t},v_{t+1}\}$ with $v_{t}\cdot v_{t}=n+1$.
Thus, by Lemma \ref{T_emb}, we have that $\varphi(\Lambda_{2})^{\perp}=\langle S_{2}\rangle^{\perp}\simeq\Lambda(n/(n-1))$.
It follows that $\varphi(\Lambda_{1})=\varphi(\Lambda_{2})^{\perp}\simeq\Lambda(2^{[n-1]})=\Lambda(n/(n-1))$, which implies that $L_{1}\cong L(n,1)$.

It remains to address the case where $b(S)=2$, so that both $S_{1}$ and $S_{2}$ are bad components.
Then, by the third subcase of the proof in \cite[p. 2162]{Lisca2}, we must have that $r/s\in\Eff_{2}$, and $S_{2}$ admits a sequence of $2$-final contractions $S_{2}\searrow\cdots\searrow C$, where $C=\{v_{t-1},v_{t},v_{t+1}\}$ with $v_{t}\cdot v_{t}=3$.
By the argument used in the previous case, it follows that $L_{1}\cong L(2,1)$.
\end{proof}

\begin{proof}[Proof of Theorem \ref{thm1}]
We first show that the conditions (1)--(3) are necessary.
Set $L_{1}=L(p,q)$, $L_{2}=L(r,s)$ and $\Lambda_{1}=\Lambda(p/(p-q))$, $\Lambda_{2}=\Lambda(r/s)$, and let $W$ be a ribbon cobordism from $L_{1}$ to $L_{2}$.
By Lemma \ref{riblemma}, we obtain a full-rank isometric embedding $\varphi\colon\Lambda_{1}\oplus\Lambda_{2}\hookrightarrow\Z^{N}$ such that $\varphi(\Lambda_{1})=\varphi(\Lambda_{2})^{\perp}$.
Let $S$ denote the corresponding linear subset, so that $\langle S\rangle=\varphi(\Lambda_{1}\oplus\Lambda_{2})$.

Suppose first that $S\subset\Z^{n}$ is irreducible.
It follows from first case of the proof in \cite[p. 2160]{Lisca2} that in this case, either $L_{1}\cong L_{2}$, or that (after possibly switching to $-W$ instead of $W$) at least one of $p/q$ and $r/s$ belongs to $\Eff_{n}$, for some $n\geq 2$.
If $L_{1}\cong L_{2}$, case (1) of Theorem \ref{thm1} holds, whereas if $r/s\in\Eff_{n}$, then Proposition \ref{key} implies that $L_{1}\cong L(n,1)$, and case (2) of Theorem \ref{thm1} holds.
If $p/q\in\Eff_{n}$, then, by the main theorem of \cite{Lisca2}, either $L_{2}\cong L(n,1)$ or $r/s\in\Eff_{n}$.
In the former case, however, by definition of $\Eff_{n}$, we have that $|H_{1}(L_{1};\Z)|>|H_{1}(L_{2};\Z)|$, which contradicts Lemma \ref{homlemma}, and we must thus have $r/s\in\Eff_{n}$.
But then, using Proposition \ref{key} again, it follows that $L_{1}\cong L(p,q)\cong L(n,1)$, which is a contradiction, since $n/1\notin\Eff_{n}$ for any $n\geq 2$.

It remains to deal with the case where $S\subset\Z^{n}$ is reducible.
In this case we can write $S=S_{1}\cup S_{2}$, where $S_{1}$ and $S_{2}$ are the maximal irreducible linear subsets that correspond to $\varphi(\Lambda_{1})$ and $\varphi(\Lambda_{2})$, respectively.
Indeed, by Lemma \ref{reducible}, $S$ cannot decompose into three or more maximal irreducible subsets.
Using the irreducibility of $S_{1}$ and the fact that $\varphi(\Lambda_{1})=\varphi(\Lambda_{2})^{\perp}$, it follows that $\varphi(\Lambda_{1})\subset\Z^{N_{1}}$ is a primitive full-rank sublattice, where $\Z^{N_{1}}=\langle e_{i}\mid v\cdot e_{i}\neq 0\text{ for some }v\in S_{1}\rangle\subset\Z^{N}$.
By Remark \ref{rem_primitivity}, we thus have that $\varphi(\Lambda_{1})\cong\Z^{N_{1}}$, which implies that $L_{1}\cong S^{3}$ and, consequently, $L_{2}\cong L(r,s)$ with $r/s\in\Err$.
Thus, case (3) of Theorem \ref{thm1} holds.

Conversely, suppose that $L_{1},L_{2}$ are lens spaces such that case (1), (2) or (3) holds.
In case (1), we can choose the product cobordism $L_{1}\times [0,1]$ to verify that $L_{1}\leq L_{2}$.
In cases (2) and (3), $L_{1}\leq L_{2}$ follows from \cite[Lemma 3.5]{Lisca2} and \cite[Theorem 1.2]{Lisca1}, respectively.
\end{proof}

\section{The proof of Theorem \ref{thm2}}\label{proof2}

In this section, we prove Theorem \ref{thm2}.
While to a large extent it is a consequence of Theorem \ref{thm1}, we will need the following additional result.

\begin{proposition}\label{2summands}
Let $Y=L_{1}\#L_{2}$, where $L_{i}$ is a lens space that does not bound a rational homology ball, $i=1,2$.
If $Y$ bounds a rational homology ball, then $Y$ must be (possibly orientation-reversingly) homeomorphic to one of the following:
\begin{enumerate}
\item $L(p,p-q)\# L(p,q)$, $p/q>1$;
\item $L(n,n-1)\# L(p,q)$, $p/q\in\Eff_{n}$ for some $n\geq 2$;
\item $L(p_{1},p_{1}-q_{1})\# L(p_{2},q_{2})$, $p_{i}/q_{i}\in\Eff_{n}$, $i=1,2$, for some $n\geq 2$; or
\item $L(p_{1},q_{1})\# L(p_{2},q_{2})$, $p_{i}/q_{i}\in\Eff_{2}$, $i=1,2$.
\end{enumerate}
Conversely, if $Y$ is homeomorphic to one of the manifolds in (1)--(4), then $Y$ bounds a \emph{ribbon} rational homology ball.
\end{proposition}

\begin{proof}
The first half of the statement is an immediate consequence of the main theorem of \cite{Lisca2}.

Conversely, suppose that $Y\cong L(p,p-q)\# L(p,q)$, $p/q>1$, and note that in this case $Y$ is the double cover of $S^{3}$ branched along the link $K(p,p-q)\#K(p,q)=\overline{K(p,q)}\#K(p,q)$.
As in the proof of Corollary \ref{cor2}, we have that $\overline{K(p,q)}\#K(p,q)$ bounds a (possibly disconnected) properly embedded surface $C\subset B^{4}$ with $\chi(C)=1$, and such that $C$ has no local maxima with respect to the radial distance function on $B^{4}$.
It follows that $Y$ bounds a ribbon rational homology ball.
The remaining cases are handled similarly (cf. also the proof of Corollary \ref{cor2}).
\end{proof}

\begin{proof}[Proof of Theorem \ref{thm2}]
Let $W$ be a ribbon cobordism from $Y_{1}=L_{1}\#\cdots\#L_{I}$ to $Y_{2}=M_{1}\#\cdots\#M_{J}$, where the $L_{1},\dots,L_{I}$ and $M_{1},\dots,M_{J}$ are lens spaces.
By Lemma \ref{riblemma}, we obtain a full-rank isometric embedding of linear lattices
$$
\varphi\colon\Lambda_{1}\oplus\cdots\oplus\Lambda_{I}\oplus\Mu_{1}\oplus\cdots\oplus\Mu_{J}\hookrightarrow\Z^{N},
$$
such that
\begin{equation}\label{perpsum}
\varphi(\Mu_{1}\oplus\cdots\oplus\Mu_{J})^{\perp}=\varphi(\Mu_{1})^{\perp}\oplus\cdots\oplus\varphi(\Mu_{J})^{\perp}=\varphi(\Lambda_{1})\oplus\cdots\oplus\varphi(\Lambda_{I}).
\end{equation}
Here, $\Lambda_{i}$ and $\Mu_{j}$ correspond to $L_{i}$ and $M_{j}$, respectively, $i=1,\dots,I$, $j=1,\dots,J$.
As in the proof of Proposition \ref{2summands}, it suffices to determine the stable isometry types of the lattices $\Lambda_{1},\dots,\Lambda_{I},\Mu_{1},\dots,\Mu_{J}$.
After possibly replacing $W$ by $-W$, we may apply Lemma \ref{reducible}, so that $\varphi$ decomposes as $\varphi=\varphi_{1}\oplus\widetilde{\varphi}$, where $\varphi_{1}\colon\Lambda\to\Z^{N_{1}}\subset\Z^{N}$ is a full-rank isometric embedding into a direct summand of $\Z^{N}$, such that the corresponding linear subset $S_{1}\subset\Z^{N_{1}}$ is irreducible, and $\Lambda$ is one of the following:
\begin{enumerate}[label=(\roman*)]
\item $\Lambda_{i}$ for some $i\in\{1,\dots,I\}$;
\item $\Lambda_{i}\oplus\Lambda_{i'}$, for some $i,i'\in\{1,\dots,I\}$, $i\neq i'$;
\item $\Mu_{j}$ for some $j\in\{1,\dots,J\}$;
\item  $\Mu_{j}\oplus\Mu_{j'}$, for some $j,j'\in\{1,\dots,J\}$, $j\neq j'$; or
\item  $\Lambda_{i}\oplus\Mu_{j}$, for some $i\in\{1,\dots,I\}$, $j\in\{1,\dots,J\}$.
\end{enumerate}

For cases (i) and (ii), note that, by \eqref{perpsum}, $\varphi(\Lambda)$ is a direct summand of the primitive sublattice $\varphi(\Mu_{1}\oplus\cdots\oplus\Mu_{J})^{\perp}\subset\Z^{N}$.
Thus, $\varphi(\Lambda)\subset\Z^{N_{1}}$ is a full-rank primitive sublattice, which, by Remark \ref{rem_primitivity}, implies that $\Lambda\cong\Z^{N_{1}}\simeq 0$.
It follows that $L_{i}\cong S^{3}$ in case (i), and $L_{i}\#L_{i'}\cong S^{3}$ in case (ii), so we may discard these connected summands from $Y_{1}$.

In case (iii), we have a full-rank isometric embedding $\varphi_{1}\colon\Mu_{j}\to\Z^{N_{1}}$, which by the main theorem of \cite{Lisca1} implies that $M_{j}\cong L(p,q)$ with $p/q\in\Err$, and case (3) of Theorem \ref{thm2} holds.

Similarly, in case (iv) it follows that $M_{j}\#M_{j'}\cong L(p_{1},q_{1})\#L(p_{2},q_{2})$, where $L(p_{1},q_{1})\#L(p_{2},q_{2})$ bounds a rational homology ball, but, since $S_{1}$ is irreducible, $L(p_{k},q_{k})$ does not, $k=1,2$.
By Proposition \ref{2summands}, $M_{j}\#M_{j'}$ bounds a ribbon rational homology ball, and must be homeomorphic to one of the manifolds listed there, so one of the cases (4)--(7) from Theorem \ref{thm2} holds.

Finally, in case (v), we have that $\varphi_{1}\colon\Lambda_{i}\oplus\Mu_{j}\hookrightarrow\Z^{N_{1}}$ is a full-rank isometric embedding which, by \eqref{perpsum} again, has the property that $\varphi_{1}(\Lambda_{i})=\varphi_{1}(\Mu_{j})^{\perp}$.
Since $S_{1}$ is irreducible, we can apply the first half of the proof of Theorem \ref{thm1} to $\varphi_{1}$ to conclude that case (1) or case (2) of Theorem \ref{thm2} must hold.

We can now apply the above procedure to $\widetilde{\varphi}$ and then iterate it, where at each step we may have to replace e.g. $\widetilde{\varphi}$ by $\widetilde{\varphi}'$, where $\widetilde{\varphi}'$ is obtained from $\widetilde{\varphi}$ as in the statement of Lemma \ref{reducible}.
Since at each step we discard a non-zero number of direct summands of $\Lambda_{1}\oplus\cdots\oplus\Lambda_{I}\oplus\Mu_{1}\oplus\cdots\oplus\Mu_{J}$, this process terminates.
\end{proof}


\bibliographystyle{alpha}
\bibliography{ribbon_cobordisms}

\end{document}